\documentclass[10pt]{amsart}%
\usepackage{graphicx}
\usepackage{amscd, color}
\usepackage{amsmath}
\usepackage{amsfonts}
\usepackage{amssymb}%
\providecommand{\U}[1]{\protect\rule{.1in}{.1in}}
\newtheorem{theorem}{Theorem}[section]

\newtheorem{remark}[theorem]{Remark}

\newtheorem{lemma}[theorem]{Lemma}

\newtheorem{definition}[theorem]{Definition}
\numberwithin{equation}{section}

\begin{document}
\author{Daniel Pellegrino and Joedson Santos}
\address[D. Pellegrino] {Departamento de Matem\'{a}tica, Universidade Federal da
Para\'{\i}ba, 58.051-900 - Jo\~{a}o Pessoa, Brazil.}
\email{pellegrino@pq.cnpq.br and dmpellegrino@gmail.com}
\address[J. Santos] {Departamento de Matem\'{a}tica, Universidade Federal da
Para\'{\i}ba, 58.051-900 - Jo\~{a}o Pessoa, Brazil.}
\email{joedsonmat@gmail.com}
\thanks{Daniel Pellegrino was supported by CNPq. }
\thanks{Joedson Santos was supported by CNPq (Edital Universal 14/2012)}
\title[ ]{Lineability and uniformly dominated sets of summing nonlinear operators}

\begin{abstract}
In this note we prove an abstract version of a result from 2002 due to Delgado
and Pi\~{n}ero on absolutely summing operators. Several applications are
presented; some of them in the multilinear framework and some in a completely
nonlinear setting. In a final section we investigate the size of the set of
non uniformly dominated sets of linear operators under the point of view of lineability.

\end{abstract}
\maketitle

\section{Introduction}

Let $X,Y$ be Banach spaces over a fixed scalar field
$\mathbb{K}=\mathbb{R}$ or $\mathbb{C}$. By $B_{X}=\{x\in
X:\left\Vert x\right\Vert \leq1\}$ we shall denote its closed unit
ball and $X^{\ast}$ shall be the topological dual of $X$. If $1\leq
p<\infty,$ a linear operator $T:X\rightarrow Y$ is said to be
absolutely $p$-summing if there exists a
constant $C_{T}\geq0$ such that%

\begin{equation}
\left(
{\displaystyle\sum\limits_{i=1}^{n}}
\left\Vert T(x_{i})\right\Vert ^{p}\right)  ^{1/p}\leq C_{T}\sup_{\varphi\in
B_{X^{\ast}}}\left(
{\displaystyle\sum\limits_{i=1}^{n}}
|\varphi(x_{i})|^{p}\right)  ^{1/p} \label{s}%
\end{equation}
for every finite set $\{x_{1},...,x_{n}\}\subset X$. For details, we refer to
the classical book by Diestel, Jarchow, Tonge \cite{djt}.

The class of absolutely $p$-summing linear operators from $X$ to $Y$ will be
represented, as it is usual, by $\Pi_{p}\left(  X,Y\right)  $ and the infimum
of all $C_{T}$ that satisfy the above inequalities defines a norm on
\ $\Pi_{p}\left(  X,Y\right)  $, denoted by $\pi_{p}(T)$.

The Pietsch Domination Theorem is one of the most important results of the
theory of summing operators and makes a surprising bridge linking Measure
Theory and summing operators. It asserts that a continuous linear operator
$T\colon X\longrightarrow Y$ between Banach spaces is absolutely $p$-summing
if and only if there is a constant $C_{T}\geq0$ and a Borel probability
measure $\mu$ on the closed unit ball of the dual of $X,$ $\left(  B_{X^{\ast
}},\sigma(X^{\ast},X)\right)  ,$ such that%
\begin{equation}
\left\Vert T(x)\right\Vert \leq C_{T}\cdot\left(  \int_{B_{X^{\ast}}%
}\left\vert \varphi(x)\right\vert ^{p}d\mu(\varphi)\right)  ^{\frac{1}{p}}
\label{gupdt}%
\end{equation}
for every $x\in X$.

Let us recall that a subset $\mathcal{M}$ of $\Pi_{p}\left(  X,Y\right)  $ is
called \emph{uniformly dominated} if there exists a positive Radon measure
$\mu$ defined on the compact space $\left(  B_{X^{\ast}},\sigma(X^{\ast
},X)\right)  $ such that
\begin{equation}
\left\Vert T(x)\right\Vert ^{p}\leq\int_{B_{X^{\ast}}}\left\vert
\varphi(x)\right\vert ^{p}d\mu(\varphi) \label{gupdt2}%
\end{equation}
for all $x\in X$ and all $T\in\mathcal{M}$. The main result of \cite{DP} is a
characterization of uniformly dominated sets when $Y$ is a Banach space that
has no finite cotype.

\begin{theorem}
\label{dp} (Delgado and Pi\~{n}ero \cite{DP}) Let $1\leq p<\infty$, let $X$ be
a Banach space, $Y$ be a Banach space that has no finite cotype, and
$\mathcal{M}\subset\Pi_{p}\left(  X,Y\right)  $. The following statements are equivalent:

(a) $\mathcal{M}$ is uniformly dominated.

(b) There is a constant $C>0$ such that, for every $\{x_{1},...,x_{n}\}
\subset X$ and $\{T_{1},...,T_{n}\} \subset\mathcal{M}$, there exists an
operator $T\in\Pi_{p}\left(  X,Y\right)  $ satisfying $\pi_{p}(T)\leq C$ and
\[
\left\Vert T_{i}(x_{i})\right\Vert \leq\left\Vert T(x_{i})\right\Vert
,\ i=1,...,n.
\]

\end{theorem}

In this note we prove an abstract version of Theorem \ref{dp} and, as
particular cases, we obtain versions of Theorem \ref{dp} in quite different
contexts. Some of our arguments are essentially an abstraction of the results
of \cite{DP}. We also estimate the size of the non uniformly dominated sets in
the sense of the theory of lineability. For details on lineability we refer to
\cite{aron, bernal} and the references therein.

\section{Preliminaries}

In the recent years a series of works (\cite{jfa, BPRn, psjmaa, joed, adv})
related to the Pietsch Domination Theorem have shown that this cornerstone of
the theory of summing operators in fact needs almost no linear structure to be
valid (see Theorem \ref{tu} below); this new panorama of the subject has been
proved to be useful in different frameworks (see, for instance, \cite{achour,
AB, BS, Janilson, CDo, RP}).

Let $X$, $Y$ and $E$ be (arbitrary) non-void sets, $\mathcal{H}\left(
X;Y\right)  $ be a non-void family of mappings from $X$ to $Y$, $G$ be a
Banach space and $K$ be a compact Hausdorff topological space. Let
\[
R\colon K\times E\times G\longrightarrow\lbrack0,\infty)~\text{and}%
\mathrm{~}S\colon\mathcal{H}\left(  X;Y\right)  \times E\times
G\longrightarrow\lbrack0,\infty)
\]
be arbitrary mappings and $1\leq p<\infty$. According to \cite{BPRn, psjmaa} a
mapping $f\in\mathcal{H}\left(  X;Y\right)  $ is $RS$-abstract $p$-summing if
there is a constant $C_{f}\geq0$ such that
\begin{equation}
\left(  \sum_{i=1}^{m}S(f,x_{i},b_{i})^{p}\right)  ^{\frac{1}{p}}\leq
C_{f}\sup_{\varphi\in K}\left(  \sum_{i=1}^{m}R\left(  \varphi,x_{i}%
,b_{i}\right)  ^{p}\right)  ^{\frac{1}{p}}, \label{33M}%
\end{equation}
for all $x_{1},\ldots,x_{m}\in E,$ $b_{1},\ldots,b_{m}\in G$ and
$m\in\mathbb{N}$. We define
\[
{\mathcal{H}}_{RS,p}(X;Y)=\left\{  f\in\mathcal{H}\left(  X;Y\right)  :\text{
}f\text{ is }RS\text{-abstract }p\text{-summing}\right\}
\]
and the infimum of the $C_{f}$'s satisfying (\ref{33M}) is denoted by
$\pi_{_{RS,p}}\left(  f\right)  .$

Suppose that $R$ is such that the mapping
\begin{equation}
R_{x,b}\colon K\longrightarrow\lbrack0,\infty)~\text{defined by}%
~R_{x,b}(\varphi)=R(\varphi,x,b) \label{ll9}%
\end{equation}
is continuous for every $x\in E$ and $b\in G$. The formulation of the Pietsch
Domination Theorem from \cite{psjmaa} reads as follows:

\begin{theorem}
[Abstract Pietsch Domination Theorem (\cite{BPRn, psjmaa})]\label{tu} Suppose
that $S$ is arbitrary, $R$ satisfies \textbf{(}\ref{ll9}\textbf{)} and let
$1\leq p<\infty$. A map $f\in\mathcal{H}\left(  X;Y\right)  $ is $RS$-abstract
$p$-summing if and only if there is a constant $C_{f}\geq0$ and a Borel
probability measure $\mu$ on $K$ such that%
\begin{equation}
S(f,x,b)\leq C_{f}\left(  \int_{K}R\left(  \varphi,x,b\right)  ^{p}%
d\mu\right)  ^{\frac{1}{p}} \label{olk}%
\end{equation}
for all $x\in E$ and $b\in G.$\newline
\end{theorem}

In the next section, we present the main result of this note, the
general version for Theorem \ref{dp} which, as the Abstract Pietsch
Domination Theorem, does not need much of the linear framework.

\begin{remark}
In the definition of $R$ we can have $R\colon K\times E\times
G\longrightarrow\mathbb{K}$ (with $\mathbb{K}$ instead of
$[0,\infty)$) and when needed we replace $R$ by $|R|$. In this case
we say $|R|S$-abstract $p$-summing.
\end{remark}

\begin{remark}
We stress that the constants from (\ref{33M}) and (\ref{olk}) can be chosen to
be the same.
\end{remark}

\section{Main Result}

The next definition is an abstract disguise of the notion of uniformly
dominated operators presented in the introduction.

\begin{definition}
A subset $\mathcal{M}$ of $\mathcal{H}_{RS,p}(X;Y)$ is \emph{uniformly
dominated} if there exists a positive Radon measure $\mu$ defined on the
compact space $K$ such that
\begin{equation}
S(f,x,b)^{p}\leq\int_{K}R\left(  \varphi,x,b\right)  ^{p}d\mu(\varphi)
\label{uda}%
\end{equation}
for all $f\in\mathcal{M}$, $x\in X$ and $b\in G.$
\end{definition}

The following lemma is somewhat expected and simple, but useful.

\begin{lemma}
\label{gol} A subset $\mathcal{M}$ of $\mathcal{H}_{RS,p}(X;Y)$ is uniformly
dominated if and only if there is a constant $C\geq0$ such that%

\begin{equation}
\left(  \sum_{i=1}^{m}S(f_{i},x_{i},b_{i})^{p}\right)  ^{\frac{1}{p}}\leq
C\sup_{\varphi\in K}\left(  \sum_{i=1}^{m}R\left(  \varphi,x_{i},b_{i}\right)
^{p}\right)  ^{\frac{1}{p}}, \label{111}%
\end{equation}
for all $\{f_{1},...,f_{m}\}\subset\mathcal{M}$, $\{x_{1},...,x_{m}\}\subset
E$ and $\{b_{1},\ldots,b_{m}\}\subset G$.
\end{lemma}

\begin{proof}
Suppose that $\mathcal{M}$ is a uniformly dominated set of $\mathcal{H}%
_{RS,p}(X;Y)$. Given $\{f_{1},...,f_{m}\}\subset\mathcal{M}$, $\{x_{1}%
,...,x_{m}\}\subset E$ and $\{b_{1},\ldots,b_{m}\}\subset G$ there is a
positive Radon measure $\mu$ such that%

\begin{align*}
\underset{i=1}{\overset{m}{%
{\displaystyle\sum}
}}S(f_{i},x_{i},b_{i})^{p}  &  \leq\underset{i=1}{\overset{m}{%
{\displaystyle\sum}
}}\int_{K}R\left(  \varphi,x_{i},b_{i}\right)  ^{p}d\mu(\varphi)\\
&  =\int_{K}\underset{i=1}{\overset{m}{%
{\displaystyle\sum}
}}R\left(  \varphi,x_{i},b_{i}\right)  ^{p}d\mu(\varphi)\\
&  \leq\mu(K)\cdot\sup_{\varphi\in K}\sum_{i=1}^{m}R\left(  \varphi
,x_{i},b_{i}\right)  ^{p}.
\end{align*}

Conversely, given $f\in\mathcal{M}$, $\{x_{1},...,x_{m}\}\subset E$ and
$\{b_{1},\ldots,b_{m}\}\subset G$ there is, by hypothesis, a constant $C\geq0$
(not depending on $f$) so that
\begin{equation}
\left(  \sum_{i=1}^{m}S(f,x_{i},b_{i})^{p}\right)  ^{\frac{1}{p}}\leq
C\sup_{\varphi\in K}\left(  \sum_{i=1}^{m}R\left(  \varphi,x_{i},b_{i}\right)
^{p}\right)  ^{\frac{1}{p}}. \label{34M}%
\end{equation}
Then $f$ is $RS$-abstract $p$-summing and by Theorem \ref{tu} there is a Borel
probability measure $\mu$ on $K$ such that%
\[
S(f,x,b)\leq C\left(  \int_{K}R\left(  \varphi,x,b\right)  ^{p}d\mu\right)
^{\frac{1}{p}}%
\]
for all $x\in E$ and $b\in G.$ Now, considering the positive Radon measure
$\overline{\mu}=C^{p}\mu$ we complete the proof.
\end{proof}

Henceforth $\frac{1}{p}+\frac{1}{p^{\prime}}=1$ and we suppose that $Y$ is a
Banach space with no finite cotype (for details on cotype we refer to
\cite[Theorem 14.1]{djt}). For this reason, we know that $Y$ contains
$\ell_{\infty}^{n}$'s uniformly (see \cite{djt}). By \cite{James}, for every
$\varepsilon>0$ and $n\in\mathbb{N}$, there is an isomorphism $J_{n}$ from
$\ell_{\infty}^{n}$ onto a subspace of $Y$ satisfying $\left\Vert J_{n}%
^{-1}\right\Vert =1$ and $\left\Vert J_{n}\right\Vert \leq(1+\varepsilon)$ for
all $n\in\mathbb{N}$. We define $y_{j}:=J_{n}e_{j}$, where $e_{j}$ are the
canonical vectors of $\ell_{\infty}^{n}$. From now on, for every $g_{j}\in
L_{p^{\prime}}(K,\mu)$, with $\left\Vert g_{j}\right\Vert _{p^{\prime}}=1$,
$j=1,...,n$ we define $f\colon X\longrightarrow Y$ by%

\begin{equation}
f(x)=\underset{j=1}{\overset{n}{%
{\displaystyle\sum}
}}\left(  \int_{K}R\left(  \varphi,x,0\right)  g_{j}(\varphi)d\mu
(\varphi)\right)  y_{j}. \label{jo3}%
\end{equation}
Henceforth we also suppose that $R$ is constant in $G$; however this
restriction causes no loss for our purposes. From now on we also suppose that
$S(g,x,b)=\left\Vert g(x)\right\Vert $, and this hypothesis is in fact
satisfied in all forthcoming applications (see Section 4). The aforementioned
hypotheses shall be weakened in our final application in Section 5.

\begin{lemma}
\label{rsa} $f$ belongs to $\mathcal{H}_{RS,p}(X;Y)$, with $\pi_{_{RS,p}%
}\left(  f\right)  \leq\mu(K)^{1/p}.$
\end{lemma}

\begin{proof}
Given $x\in X$, by one of the versions of the Hahn--Banach Theorem, we have%

\begin{align*}
\left\Vert f(x)\right\Vert  &  =\sup_{y^{\ast}\in B_{Y^{\ast}}}\left\vert
\langle y^{\ast},f(x)\rangle\right\vert \\
&  =\sup_{y^{\ast}\in B_{Y^{\ast}}}\left\vert \langle y^{\ast},\underset
{j=1}{\overset{n}{%
{\displaystyle\sum}
}}\left(  \int_{K}R\left(  \varphi,x,0\right)  g_{j}(\varphi)d\mu
(\varphi)\right)  y_{j}\rangle\right\vert \\
&  \leq\sup_{y^{\ast}\in B_{Y^{\ast}}}\underset{j=1}{\overset{n}{%
{\displaystyle\sum}
}}\left(  \int_{K} R\left(  \varphi,x,0\right)
|g_{j}(\varphi)|d\mu(\varphi)\right)  \left\vert \langle y^{\ast},y_{j}%
\rangle\right\vert \\
&  \leq\sup_{y^{\ast}\in B_{Y^{\ast}}}\underset{j=1}{\overset{n}{%
{\displaystyle\sum}
}}\left(  \int_{K} R\left(  \varphi,x,0\right) ^{p}%
d\mu(\varphi)\right)  ^{1/p}\left(  \int_{K}|g_{j}(\varphi)|^{p^{\prime}}%
d\mu(\varphi)\right)  ^{1/p^{\prime}}\left\vert \langle y^{\ast},y_{j}%
\rangle\right\vert \\
&  =\left(  \int_{K} R\left(  \varphi,x,0\right)
^{p}d\mu(\varphi)\right)  ^{1/p}\sup_{y^{\ast}\in
B_{Y^{\ast}}}\underset
{j=1}{\overset{n}{%
{\displaystyle\sum}
}}\left\vert \langle y^{\ast},y_{j}\rangle\right\vert \\
&  \leq\left(  \int_{K}R\left(  \varphi,x,0\right)
^{p}d\mu(\varphi)\right)  ^{1/p}\left\Vert J_{n}^{\ast}\right\Vert \\
&  =\left(  \int_{K} R\left(  \varphi,x,0\right)
^{p}d\mu(\varphi)\right)  ^{1/p}\left\Vert J_{n}\right\Vert \\
&  \leq\left(  \int_{K} R\left(  \varphi,x,0\right)
^{p}d\mu(\varphi)\right)  ^{1/p}(1+\varepsilon).
\end{align*}

Considering the probability measure $\overline{\mu}=\mu/\mu(K)$, we get%
\begin{equation}
S(f,x,b)\leq\left(  1+\varepsilon\right)  \mu(K)^{1/p}\cdot\left(
\int _{K} R\left(  \varphi,x,0\right) ^{p}d\overline{\mu
}(\varphi)\right)  ^{1/p}.\label{fla}%
\end{equation}

Hence $f\in\mathcal{H}_{RS,p}(X;Y)$. Since $\varepsilon>0$ is arbitrary we
obtain $\pi_{_{RS,p}}\left(  f\right)  \leq\mu(K)^{1/p}.$
\end{proof}

Now we state and prove our main result.

\begin{theorem}
\label{dpa} Let $Y$ is a Banach space with no finite cotype, $\mathcal{M}%
\subset\mathcal{H}_{RS,p}(X;Y)$ and $1\leq p<\infty$. The following statements
are equivalent:

(a) $\mathcal{M}$ is uniformly dominated.

(b) There is a constant $C>0$ such that, for every $\{x_{1},...,x_{n}\}\subset
X$ and $\{f_{1},...,f_{n}\}\subset\mathcal{M}$, there exists an operator
$f\in\mathcal{H}_{RS,p}(X;Y)$ satisfying $\pi_{_{RS,p}}\left(  f\right)  \leq
C$ and%
\[
\left\Vert f_{i}(x_{i})\right\Vert \leq\left\Vert f(x_{i})\right\Vert
,\ i=1,...,n.
\]

\end{theorem}

\begin{proof}
$(a)\Rightarrow(b)$ By hypothesis, there exists a positive Radon measure $\mu$
such that (recall that we are supposing that $R$ is constant in $G$),
\begin{equation}
\left\Vert u(x)\right\Vert \leq\left(  \int_{K}R\left(  \varphi,x,0\right)
^{p}d\mu(\varphi)\right)  ^{1/p}\label{jj}%
\end{equation}
for all $u\in\mathcal{M}$, $x\in X.$

Given $\{f_{1},...,f_{n}\}\subset\mathcal{M}$, $\{x_{1},...,x_{n}\}\subset E$,
by (\ref{jj}) we have
\begin{equation}
\left\Vert f_{i}(x_{i})\right\Vert \leq\left(  \int_{K}R\left(  \varphi
,x_{i},0\right)  ^{p}d\mu(\varphi)\right)  ^{1/p}\ i=1,...,n. \label{jo}%
\end{equation}
For every $i=1,...,n$, take $g_{i}\in L_{p^{\prime}}(\mu)$ such that
$\left\Vert g_{i}\right\Vert _{p^{\prime}}=1$ and
\begin{equation}
\left(  \int_{K}R\left(  \varphi,x_{i},0\right)  ^{p}d\mu(\varphi)\right)
^{1/p}=\int_{K}R\left(  \varphi,x_{i},0\right)  g_{i}(\varphi)d\mu(\varphi).
\label{jo1}%
\end{equation}
From (\ref{jo}) and (\ref{jo1}), we obtain
\begin{equation}
\left\Vert f_{i}(x_{i})\right\Vert \leq\int_{K}R\left(  \varphi,x_{i}%
,0\right)  g_{i}(\varphi)d\mu(\varphi),\ i=1,...,n. \label{jo2}%
\end{equation}

We now define $y_{i}^{\ast}=e_{i}^{\ast}\circ J_{n}^{-1}$, where $e_{i}^{\ast
}$ are the canonical vectors of $\left(  \ell_{\infty}^{n}\right)  ^{\ast
}\simeq\ell_{1}^{n}$. Notice that $\left\Vert y_{i}^{\ast}\right\Vert \leq1$
for $i=1,...,n$. We also denote by $y_{i}^{\ast}$ a Hahn-Banach extension of
$e_{i}^{\ast}\circ J_{n}^{-1}$ to $Y$. Thus, using $f$ defined in (\ref{jo3})
and Lemma \ref{rsa}, we obtain%

\begin{align*}
\left\Vert f(x_{i})\right\Vert  &  =\sup_{y^{\ast}\in B_{Y^{\ast}}}\left\vert
\langle y^{\ast},f(x_{i})\rangle\right\vert \\
&  \geq\left\vert \langle y_{i}^{\ast},f(x_{i})\rangle\right\vert \\
&  =\left\vert \langle y_{i}^{\ast},\underset{j=1}{\overset{n}{%
{\displaystyle\sum}
}}\left(  \int_{K}R\left(  \varphi,x_{i},0\right)  g_{j}(\varphi)d\mu
(\varphi)\right)  y_{j}\rangle\right\vert \\
&  =\left\vert \underset{j=1}{\overset{n}{%
{\displaystyle\sum}
}}\left(  \int_{K}R\left(  \varphi,x_{i},0\right)  g_{j}(\varphi)d\mu
(\varphi)\right)  \langle y_{i}^{\ast},y_{j}\rangle\right\vert \\
&  =\left\vert \underset{j=1}{\overset{n}{%
{\displaystyle\sum}
}}\left(  \int_{K}R\left(  \varphi,x_{i},0\right)  g_{j}(\varphi)d\mu
(\varphi)\right)  \langle e_{i}^{\ast}\circ J_{n}^{-1},J_{n}e_{j}%
\rangle\right\vert \\
&  =\left\vert \underset{j=1}{\overset{n}{%
{\displaystyle\sum}
}}\left(  \int_{K}R\left(  \varphi,x_{i},0\right)  g_{j}(\varphi)d\mu
(\varphi)\right)  \langle e_{i}^{\ast},e_{j}\rangle\right\vert \\
&  =\int_{K}R\left(  \varphi,x_{i},0\right)  g_{i}(\varphi)d\mu(\varphi)\\
&  \overset{(\ref{jo2})}{\geq}\left\Vert f_{i}(x_{i})\right\Vert
\end{align*}

$(b)\Rightarrow(a)$ It follows from Lemma \ref{gol}. In fact, by hypotheses,
there is a constant $C>0$ and an operator $f\in\mathcal{H}_{RS,p}(X;Y)$
satisfying $\pi_{_{RS,p}}\left(  f\right)  \leq C$ and
\[
\left\Vert f_{i}(x_{i})\right\Vert \leq\left\Vert f(x_{i})\right\Vert
,\ i=1,...,n
\]
for all $\{x_{1},...,x_{n}\}\subset X$ and $\{f_{1},...,f_{n}\}\subset
\mathcal{M}$. Thus%

\[
\left(  \sum_{i=1}^{n}\left\Vert f_{i}(x_{i})\right\Vert ^{p}\right)
^{\frac{1}{p}}\leq\left(  \sum_{i=1}^{n}\left\Vert f(x_{i})\right\Vert
^{p}\right)  ^{\frac{1}{p}}\leq C\sup_{\varphi\in K}\left(  \sum_{i=1}%
^{n}R\left(  \varphi,x_{i},b_{i}\right)  ^{p}\right)  ^{\frac{1}{p}},
\]
for all $\{f_{1},...,f_{n}\}\subset\mathcal{M}$, $\{x_{1},...,x_{n}\}\subset
X$ and $\{b_{1},\ldots,b_{n}\}\subset G$. By Lemma \ref{gol}, $\mathcal{M}$ is
uniformly dominated.
\end{proof}

\section{Consequences of Theorem \ref{dpa}}

In this section we apply Theorem \ref{dpa} to characterize uniformly dominated
sets in several usual classes of $RS$-abstract summing mappings when $Y$ is a
Banach space that has no finite cotype.

\subsection{ Absolutely $p$-summing linear operators}

Observe that a continuous linear operator $T:X\rightarrow Y$ is
absolutely $p$-summing if and only if it is $|R|S$-abstract
$p$-summing with
\[
E=X\text{ and }G=\mathbb{K}%
\]
and $K=B_{X^{\ast}}$, with the weak star topology, ${\mathcal{H}%
}(X;Y)={\mathcal{L}}(X;Y)$ and $R$ and $S$ are defined by:
\[
R\colon B_{X^{\ast}}\times X\times\mathbb{K}\longrightarrow\mathbb{K}%
~,~R(\varphi,x,b)=\varphi(x)
\]%
\[
S\colon{\mathcal{L}}(X;Y)\times X\times\mathbb{K}\longrightarrow
\lbrack0,\infty)~,~S(T,x,b)=\left\Vert T(x)\right\Vert .
\]
We can see that the hypotheses under $R$ and $S$ are
straightforwardly satisfied, so Theorem \ref{dpa} recovers Theorem
\ref{dp}.

\subsection{ Strongly $p$-summing multilinear mappings}

Strongly $p$-summing multilinear mappings were introduced by Dimant
\cite{Dimant}. Let $X_{1},...,X_{n}$ be Banach spaces. A continuous $n$-linear
mapping $T\colon X_{1}\times\cdots\times X_{n}\longrightarrow Y$ is strongly
$p$-summing if there is a constant $C>0$ such that
\begin{equation}
\left(  \sum_{i=1}^{m}\left\Vert T\left(  x_{i}^{1},...,x_{i}^{n}\right)
\right\Vert ^{p}\right)  ^{\frac{1}{p}}\leq C\sup_{\varphi\in B_{\mathcal{L}%
(X_{1},...,X_{n})}}\left(  \sum_{i=1}^{m}|\varphi\left(  x_{i}^{1}%
,...,x_{i}^{n}\right)  |^{p}\right)  ^{\frac{1}{p}}, \label{33Ma}%
\end{equation}
for every $m\in\mathbb{N}$ and $x_{i}^{l}\in X_{l},$ $i=1,...,m,$ $l=1,...,n,$
where $\mathcal{L}(X_{1},...,X_{n})$ is the space of all continuous $n$-linear
forms on $X_{1}\times\cdots\times X_{n}$.

Note that by choosing the parameters $E=X_{1}\times\cdots\times X_{n}$,
$K=B_{(X_{1}\widehat{\otimes}_{\pi}\cdots\widehat{\otimes}_{\pi}X_{n})^{\ast}%
}$, $G=\mathbb{K}$, $\mathcal{H}=\mathcal{L}(X_{1},...,X_{n};Y)$ and $R$ and
$S$ are defined by:
\[
R\colon B_{(X_{1}\widehat{\otimes}_{\pi}\cdots\widehat{\otimes}_{\pi}%
X_{n})^{\ast}}\times(X_{1}\times\cdots\times X_{n})\times\mathbb{K}%
\longrightarrow\mathbb{K}~,~R(\varphi,(x^{1},...,x^{n}),b)=\varphi
(x^{1}\otimes...\otimes x^{n})
\]%
\[
S\colon\mathcal{L}(X_{1},...,X_{n};Y)\times(X_{1}\times\cdots\times
X_{n})\times\mathbb{K}\longrightarrow\lbrack0,\infty)~,~S(T,(x^{1}%
,...,x^{n}),b)=\left\Vert T(x^{1},...,x^{n})\right\Vert ,
\]
we can easily conclude that $T:X_{1}\times\cdots\times
X_{n}\rightarrow Y$ is strongly $p$-summing if and only if $T$ is
$|R|S$-abstract $p$-summing.

It is immediate to verify that the hypotheses needed for $R$ and $S$
are also straightforwardly satisfied in this case. Thus, Theorem
\ref{dpa} provides a characterization of uniformly dominated set in
classes of strongly $p$-summing multilinear mappings.

\subsection{ $p$-semi-integral multilinear mappings}

The class $p$-semi-integral multilinear mappings was introduced in \cite{CD},
inspired by previous work of R. Alencar and M.C. Matos. Let $X_{1},...,X_{n}$
be Banach spaces. A continuous $n$-linear mapping $T\colon X_{1}\times
\cdots\times X_{n}\longrightarrow Y$ is $p$-semi-integral if there is a
constant $C>0$ such that
\begin{equation}
\left(  \sum_{i=1}^{m}\left\Vert T\left(  x_{i}^{1},...,x_{i}^{n}\right)
\right\Vert ^{p}\right)  ^{\frac{1}{p}}\leq C\sup_{\varphi_{j}\in
B_{X_{j}^{\ast}},j=1,...,n}\left(  \sum_{i=1}^{m}|\varphi_{1}\left(  x_{i}%
^{1}\right)  \cdots\varphi_{n}\left(  x_{i}^{n}\right)  |^{p}\right)
^{\frac{1}{p}}, \label{33Mb}%
\end{equation}
for every $m\in\mathbb{N}$ and $x_{i}^{l}\in X_{l},$ $i=1,...,m,$ $l=1,...,n$.

Choosing the parameters $E=X_{1}\times\cdots\times X_{n}$, $K=B_{X_{1}^{\ast}%
}\times\cdots\times B_{X_{n}^{\ast}}$, $G=\mathbb{K}$, $\mathcal{H}%
=\mathcal{L}(X_{1},...,X_{n};Y)$ and $R$ and $S$ are defined by:
\[
R\colon(B_{X_{1}^{\ast}}\times\cdots\times B_{X_{n}^{\ast}})\times(X_{1}%
\times\cdots\times X_{n})\times\mathbb{K}\longrightarrow\mathbb{K}%
\]
such that%

\[
R\left(  (\varphi_{1},...,\varphi_{n}),(x^{1},...,x^{n}),b\right)
=\varphi_{1}\left(  x^{1}\right)  \cdots\varphi_{n}\left(  x^{n}\right)
\]
and%

\[
S\colon\mathcal{L}(X_{1},...,X_{n};Y)\times(X_{1}\times\cdots\times
X_{n})\times\mathbb{K}\longrightarrow\lbrack0,\infty)~,~S(T,(x^{1}%
,...,x^{n}),b)=\left\Vert T(x^{1},...,x^{n})\right\Vert ,
\]
we observe that $T:X_{1}\times\cdots\times X_{n}\rightarrow Y$ is
$p$-semi-integral if and only if $T$ is $|R|S$-abstract $p$-summing.

In this case, Theorem \ref{dpa} provides a characterization of uniformly
dominated sets in this context.

\subsection{ Homogeneous mappings}

Let $X$ a Banach space. In \cite{mona} a continuous mapping $u:X\rightarrow Y$
such that%
\begin{equation}
\left\Vert u(\lambda x)\right\Vert \geq\left\vert \lambda\right\vert
\left\Vert u(x)\right\Vert \label{tq}%
\end{equation}
for all $\left(  x,\lambda\right)  $ $\in X\times\mathbb{K}$ is called $1$-subhomogeneous.

From (\cite[Theorem 2.3 (b)]{mona}) a $1$-subhomogeneous map $u$ is absolutely
$p$-summing if and only if there is a $C>0$ such that%
\begin{equation}
\left(  \sum_{j=1}^{m}\left\Vert u(x_{j})\right\Vert ^{p}\right)  ^{\frac
{1}{p}}\leq C\sup_{\varphi\in B_{X^{\ast}}}\left(  \sum_{j=1}%
^{m}\left\vert \varphi(x_{j})\right\vert ^{p}\right)  ^{\frac{1}{p}}
\label{22333333}%
\end{equation}
for every positive integer $m.$ It is simple to note that a $1$-subhomogeneous
map $u$ is absolutely $p$-summing if and only if it is $RS$-abstract
$p$-summing with
\[
E=X\text{ and }G=\mathbb{R}%
\]
and $K=B_{X^{\ast}}$, with the weak star topology, ${\mathcal{H}}(X;Y)$ is the
space of continuous homogeneous mappings from $X$ to $Y$ and $R$ and $S$ are
defined by
\[
R\colon B_{X^{\ast}}\times X\times\mathbb{R}\longrightarrow\lbrack
0,\infty)\subset\mathbb{R}~,~R(\varphi,x,b)=|\varphi(x)|
\]%
\[
S\colon{\mathcal{H}}(X;Y)\times X\times\mathbb{R}\longrightarrow
\lbrack0,\infty)~,~S(T,x,b)=\left\Vert T(x)\right\Vert .
\]

As expected, Theorem \ref{dpa} also characterizes uniformly dominated sets in
this classes of operators.

\section{Final Application: Absolutely summing arbitrary mappings}

\bigskip Let $X$ be a Banach space. Following \cite[Definition 4.1]{BPRn}, an
arbitrary mapping $u:X\rightarrow Y$ is absolutely $p$-summing at $a\in X$ if
there is a $C>0$ so that
\begin{equation}
\left(  \sum_{j=1}^{m}\left\Vert u(a+x_{j})-u(a)\right\Vert ^{p}\right)
^{\frac{1}{p}}\leq C\sup_{\varphi\in B_{X^{\ast}}}\left(  \sum_{j=1}%
^{m}\left\vert \varphi(x_{j})\right\vert ^{p}\right)  ^{\frac{1}{p}}
\label{mmu222}%
\end{equation}
for every natural number $m$ and every $x_{1},...,x_{m}\in X.$

Choosing $E=X$, $G=\mathbb{R}$, $K=B_{X^{\ast}}$, with the weak star topology,
${\mathcal{H}}(X;Y)$ being the set of maps from $X$ to $Y$ and $R$, $S$ being
defined by
\[
R\colon B_{X^{\ast}}\times X\times\mathbb{R}\longrightarrow\lbrack
0,\infty)\subset\mathbb{R}~,~R(\varphi,x,b)=|\varphi(x)|
\]%
\[
S\colon{\mathcal{H}}(X;Y)\times X\times\mathbb{R}\longrightarrow
\lbrack0,\infty)~,~S(h,x,b)=\left\Vert h(x)-h(0)\right\Vert
\]
we conclude that an arbitrary mapping $h:X\rightarrow Y$ is absolutely
$p$-summing at $0\in X$ if and only if $h$ is $RS$-abstract $p$-summing.

A simple reformulation of the hypotheses that appear just below (\ref{jo3})
allows us to prove that Theorem \ref{dpa} is valid in this more arbitrary
context, with $a=0$ (we just need to ask that $S(g,x,b)=\left\Vert
g(x)\right\Vert $ just for $g=f$ defined in (\ref{jo3}) and not for all
$g$)$.$ Besides, since
\[
f(0)=\underset{j=1}{\overset{n}{%
{\displaystyle\sum}
}}\left(  \int_{K}R\left(  \varphi,0,0\right)  g_{j}(\varphi)d\mu
(\varphi)\right)  y_{j}=\underset{j=1}{\overset{n}{%
{\displaystyle\sum}
}}\left(  \int_{K}|\varphi(0)|g_{j}(\varphi)d\mu(\varphi)\right)  y_{j}=0,
\]
we in fact have $S(f,x,b)=\left\Vert f(x)\right\Vert ;$ so the re-formulation
of Theorem \ref{dpa} also characterizes uniformly dominated sets of absolutely
summing arbitrary maps. To summarize:

\begin{theorem}
Let $1\leq p<\infty.$ Let $Y$ is a Banach space with no finite cotype and
$\mathcal{M}$ be a subset of the set of all arbitrary mappings $u:X\rightarrow
Y$ which are absolutely $p$-summing at $0$ (denoted by $\mathcal{H}%
_{RS,p}(X;Y)$). The following statements are equivalent:

(a) $\mathcal{M}$ is uniformly dominated.

(b) There is a constant $C>0$ such that, for every $\{x_{1},...,x_{n}\}\subset
X$ and $\{f_{1},...,f_{n}\}\subset\mathcal{M}$, there exists an operator
$f\in\mathcal{H}_{RS,p}(X;Y)$ satisfying $\pi_{_{RS,p}}\left(  f\right)  \leq
C$ and%
\[
\left\Vert f_{i}(x_{i})\right\Vert \leq\left\Vert f(x_{i})\right\Vert
,\ i=1,...,n.
\]

\end{theorem}

\begin{remark}
With some more technical effort it is also possible to prove the above result
for $a\neq0$ but we omit the details.
\end{remark}

\section{Lineability of $\Pi_{p}\left(  X,\ell_{\infty}\right)  \smallsetminus
\mathcal{M}$}

We say that the subset $M$ of a vector space $E$ is \emph{lineable} if
$M\cup\{0\}$ contains an infinite dimensional linear space (see \cite{aron,
bernal, seo} and the references therein).

In this section we investigate the size of the set $\Pi_{p}\left(
X,\ell_{\infty}\right)  \smallsetminus\mathcal{M}$ from the point of view of
the theory of lineability. In fact, we show that, up to the null vector,
$\Pi_{p}\left(  X,\ell_{\infty}\right)  \smallsetminus\mathcal{M}$ contains a
subspace of dimension $\mathfrak{c}$ (cardinality of the
continuum)$\mathfrak{.}$ Let us suppose that $\mathcal{M}\subset\Pi_{p}\left(
X,\ell_{\infty}\right)  $ is uniformly dominated and $\Pi_{p}\left(
X,\ell_{\infty}\right)  \smallsetminus\mathcal{M}$ is non void. If $T\in
\Pi_{p}\left(  X,\ell_{\infty}\right)  \smallsetminus\mathcal{M}$ and $\mu$
\ is a positive Radon measure defined on the compact space $\left(
B_{X^{\ast}},\sigma(X^{\ast},X)\right)  ,$ then%
\begin{equation}
\left\Vert T(x)\right\Vert ^{p}>\int_{B_{X^{\ast}}}\left\vert \varphi
(x)\right\vert ^{p}d\mu(\varphi) \label{777}%
\end{equation}
for some $x\in X$.

Now we separate the set of positive integers $\mathbb{N}$ into countably many
infinite pairwise disjoint subsets $\left(  A_{k}\right)  _{k=1}^{\infty}.$
For each positive integer $k$, let us denote%
\[
A_{k}=\left\{  a_{1}^{(k)}<a_{2}^{(k)}<\cdots\right\}
\]
and consider%
\[
\ell_{\infty}^{(k)}=\left\{  x\in\ell_{\infty}:x_{j}=0\text{ if }j\notin
A_{k}\right\}  .
\]
Now, for each fixed positive integer $k$, we define%
\[
T_{k}:X\rightarrow\ell_{\infty}^{(k)}%
\]
given by%
\[
\left(  T_{k}\left(  z\right)  \right)  _{a_{j}^{(k)}}=\left(  T\left(
z\right)  \right)  _{j}%
\]
for all positive integer $j.$ Thus, for any fixed $k$, let $v_{k}%
:X\rightarrow\ell_{\infty}$ be given by%
\[
v_{k}=i_{k}\circ T_{k},
\]
where $i_{k}:\ell_{\infty}^{(k)}\rightarrow\ell_{\infty}$ is the canonical
inclusion. It is plain that%
\[
\left\Vert v_{k}(z)\right\Vert =\left\Vert T_{k}(z)\right\Vert =\left\Vert
T(z)\right\Vert
\]
for every positive integer $k$ and $z\in X.$ Thus, each $v_{k}$ satisfies
(\ref{777}). Note also that the operators $v_{k}$ have disjoint supports and
thus the set $\left\{  v_{1},v_{2},...\right\}  $ is linearly independent.
Consider the operator%
\[
S:\ell_{1}\rightarrow\Pi_{p}\left(  X,\ell_{\infty}\right)
\]
given by%
\[
S\left(  \left(  a_{k}\right)  _{k=1}^{\infty}\right)  =%
{\textstyle\sum\limits_{k=1}^{\infty}}
a_{k}v_{k}.
\]
Since%
\begin{align*}%
{\textstyle\sum\limits_{k=1}^{\infty}}
\left\Vert a_{k}v_{k}\right\Vert  &  =%
{\textstyle\sum\limits_{k=1}^{\infty}}
\left\vert a_{k}\right\vert \left\Vert v_{k}\right\Vert \\
&  =%
{\textstyle\sum\limits_{k=1}^{\infty}}
\left\vert a_{k}\right\vert \left\Vert T\right\Vert \\
&  =\left\Vert T\right\Vert
{\textstyle\sum\limits_{k=1}^{\infty}}
\left\vert a_{k}\right\vert <\infty
\end{align*}
we conclude that $S$ is well-defined and, moreover, $S$ is linear and
injective. Since the supports of the operators $v_{k}$ are pairwise disjoint,
we conclude that $S(\ell_{1})$ satisfies (\ref{777}). In fact, suppose that
$b_{1},....,b_{m}$ are all nonzero scalars and $a_{1},...,a_{m}$ are also
scalars (and there is no loss of generality in supposing $a_{1}$ non null).
For any positive Radon measure $\mu$ we have (recall the definition of the
$v_{k}$ and that this norm is in $\ell_{\infty}$)
\begin{align*}
\left\Vert b_{1}%
{\textstyle\sum\limits_{k=1}^{\infty}}
a_{k}v_{k}(x)+\cdots+b_{m}%
{\textstyle\sum\limits_{k=1}^{\infty}}
a_{k}v_{m}(x)\right\Vert ^{p}  &  \geq\left\Vert b_{1}a_{1}v_{1}(x)\right\Vert
^{p}\\
&  =\left\vert b_{1}a_{1}\right\vert ^{p}\left\Vert v_{1}(x)\right\Vert ^{p}\\
&  >\left\vert b_{1}a_{1}\right\vert ^{p}\int_{B_{X^{\ast}}}\left\vert
\varphi(x)\right\vert ^{p}d\mu(\varphi).
\end{align*}
Now replacing $\mu$ by the positive Radon measure $\left\vert b_{1}%
a_{1}\right\vert ^{p}\mu$, denoted by $\psi$, we have%
\[
\left\Vert b_{1}%
{\textstyle\sum\limits_{k=1}^{\infty}}
a_{k}v_{k}(x)+\cdots+b_{m}%
{\textstyle\sum\limits_{k=1}^{\infty}}
a_{k}v_{m}(x)\right\Vert ^{p}>\int_{B_{X^{\ast}}}\left\vert \varphi
(x)\right\vert ^{p}d\psi(\varphi)
\]
Thus, since the map $\mu\leftrightarrow\left\vert b_{1}a_{1}\right\vert
^{p}\mu$ is a bijection in the set of positive Radon measures we have
\[
S(\ell_{1})\subset\left(  \Pi_{p}\left(  X,\ell_{\infty}\right)
\smallsetminus\mathcal{M}\right)  \cup\left\{  0\right\}
\]
and the proof is done, because $\dim\left(  \ell_{1}\right)  =\mathfrak{c}.$

\end{document}